\documentclass[12pt]{amsart}

\usepackage{url}
\usepackage{mathtools}
\usepackage{enumerate}
\usepackage{comment}
\usepackage{color}
\usepackage[margin=25mm]{geometry}  
\usepackage{eucal,mathrsfs,dsfont} 

\newtheorem{theorem}{Theorem}[section]  
\newtheorem{lemma}[theorem]{Lemma }  
\newtheorem{proposition}[theorem]{Proposition} 

\theoremstyle{definition} 
\newtheorem{definition}[theorem]{Definition} 
\numberwithin{equation}{section}

\newcommand{\eps}{\varepsilon}

\newcommand{\calL}{\mathcal{L}}

\newcommand{\calF}{\mathcal{F}}
\newcommand{\sC}{\mathcal{C}}

\newcommand{\sF}{\mathcal{F}}
\newcommand{\sX}{\mathcal{X}}

\newcommand{\calO}{\mathcal{O}}
\newcommand{\sO}{\mathcal{O}}

\newcommand{\calD}{\mathcal{D}}

\newcommand{\calT}{\mathcal{T}}
\newcommand{\calS}{\mathcal{S}}
\newcommand{\sS}{\mathcal{S}}

\newcommand{\calB}{\mathcal{B}}
\newcommand{\calV}{\mathcal{V}}

\newcommand{\sE}{\mathcal{E}} 
 
\newcommand{\sH}{\mathcal{H}}

\newcommand{\calK}{\mathcal{K}}

\newcommand{\bE}{\operatorname{\mathbb{E}}} 
\newcommand{\Pp}{\operatorname{\mathbb{P}}} 
\newcommand{\bP}{\operatorname{\mathbb{P}}} 

\newcommand{\R}{\mathbb{R}}
\newcommand{\bR}{\mathbb{R}}

\newcommand{\Z}{{\mathbb Z}}
\newcommand{\bZ}{{\mathbb Z}}

\newcommand{\prt}{\partial}
\newcommand{\PBM}{P_{\text{BM}}}
\newcommand{\EBM}{E_{\text{BM}}}
\newcommand{\ball}{\mathcal{B}}
\newcommand{\res}{\mathbf{r}}
\newcommand{\assmp}{($\clubsuit$) }

\newcommand{\ol}{\overline}
\newcommand{\wh}{\widehat}
\newcommand{\wt}{\widetilde}

\newcommand{\std}{\mathbf{p}}

\DeclareMathOperator{\dist}{dist}
\DeclareMathOperator{\Osc}{Osc}
\DeclareMathOperator{\Var}{Var}
\DeclareMathOperator{\Cov}{Cov}

\def\bone{{\bf 1}}

\newcommand{\dP}{\mathbf{d}}
\def\Sd{{\rm d_S}}

\def\qP{{P}}
\def\qE{{E}}
\def\pd{{\partial}}
\def\q{\quad}

\def\be{\begin{equation}}
\def\ee{\end{equation}}
\def\bes{\begin{equation*}}
\def\ees{\end{equation*}}

\def\om{{\omega}}
\def\half{\frac12}

\def\lam{\lambda}

\def\sK{\calK}

\def\sm{\smallskip \noindent}

\def\be{\begin{equation}}
\def\ee{\end{equation}}
\def\bes{\begin{equation*}}
\def\ees{\end{equation*}}

\def\bfP{{\bf P}}

\def\eqd{{\buildrel (d) \over {\  =\ }}}

\begin{document}

\title[Invariance principle]{Comparison of quenched and annealed invariance   
principles for random conductance model: Part II}

\author{Martin Barlow, Krzysztof Burdzy and Ad\'am Tim\'ar}

\address{Department of Mathematics, University of British Columbia, 
Vancouver, B.C., Canada V6T 1Z2}

\address{Department of Mathematics, Box 354350, University of Washington, 
Seattle, WA 98195, USA}

\address{
Bolyai Institute, University of Szeged,
Aradi v. tere 1, 6720 Szeged,
Hungary
}

\thanks{Research supported in part by NSF Grant DMS-1206276, by 
NSERC, Canada, and Trinity College, Cambridge,  and by
MTA R\'enyi "Lendulet" Groups and Graphs Research Group.}

\begin{abstract}
We show that there exists an ergodic conductance environment such that the 
weak (annealed) invariance principle holds for the corresponding  
continuous time random walk but the quenched invariance principle does not hold.
In the present paper we give a proof of the full scaling limit for the weak invariance principle, improving the
result in an earlier paper where we obtained a subsequential limit.
\end{abstract}

\maketitle  

\section{Introduction}\label{intro}

This article contains the completion of the project started in a previous paper \cite{BBT1}, 
where we proved that there exists an ergodic conductance environment such that the 
weak (annealed) invariance principle holds for the corresponding  
continuous time random walk along a subsequence but the quenched invariance principle does not hold.
In the present paper we give a proof of the full scaling limit for the weak invariance principle, improving the
result in \cite{BBT1}.
The improved result is, in a sense, a quantitative form of the invariance 
principle. The proof consists of several lemmas. Some of them are specific to our 
model but some of them have the more general character and may serve as 
technical elements for related projects. Since this paper is a continuation of 
\cite{BBT1}, we start by presenting basic notation and definitions from that paper.

Let $d\geq 2$ and let $ E_d$ be the set of all non oriented edges in the 
$d$-dimensional integer lattice, that is,  $E_d = \{e = \{x,y\}: x,y \in \Z^d, |x-y|=1\}$.
Let $\{\mu_e\}_{e\in E_d}$ be a random process with non-negative values, 
defined on some probability space $(\Omega, \calF, \Pp)$. 
The process $\{\mu_e\}_{e\in E_d}$ represents random conductances.  
We write $\mu_{xy}  = \mu_{yx} = \mu_{\{x,y\}}$ and set 
$\mu_{xy}=0$ if $\{x,y\} \notin E_d$. Set
\begin{align*}
\mu_x = \sum_y \mu_{xy}, \qquad P(x,y) = \frac{\mu_{xy}}{\mu_x},
\end{align*}
with the convention that $0/0=0$ and $P(x,y)=0$ if $\{x,y\} \notin E_d$. 
For a fixed $\omega\in \Omega$, let  
$X = \{X_t, t\geq 0, \qP^x_\omega, x \in \Z^d\}$ be the 
continuous time random walk on $\Z^d$, with transition probabilities 
$P(x,y) = P_\omega(x,y)$, and exponential waiting times with mean $1/\mu_x$. 
The corresponding expectation will be denoted $\qE_\omega^x$. 
For a fixed $\omega\in \Omega$, the generator $\calL$ of $X$ is given by
\begin{align}\label{e:Ldef}
\calL f(x) = \sum_y \mu_{xy} (f(y) - f(x)).
\end{align}
In \cite{BD} this is called the {\em variable speed random walk} (VSRW)
among the conductances $\mu_e$.
This model, of a reversible (or symmetric) random walk in a random environment, is
often called the Random Conductance Model. 

We are interested in functional Central Limit Theorems (FCLTs)
for the process $X$. Given any process $X$,
for $\eps>0$, set $X^{\eps}_t = \eps X_{t /\eps^2}$, $t\geq 0$. 
Let $\calD_T = D([0,T], \R^d)$ denote the Skorokhod space, 
and let $\calD_\infty=D([0,\infty), \R^d)$.
Write $d_S$ for the Skorokhod metric and $\calB(\calD_T)$ for the $\sigma$-field of 
Borel sets in the corresponding topology. 
Let $X$ be the canonical process on $\calD_\infty$ or $\calD_T$, $\PBM$ be Wiener 
measure on $(\calD_\infty, \calB(\calD_\infty))$  and let $\EBM$ be the 
corresponding expectation. 
We will write $W$ for a standard Brownian motion.
It will be convenient to assume that $\{\mu_e\}_{e\in E_d}$ are 
defined on a probability space $(\Omega, \sF, \bP)$, and that
$X$ is defined on $(\Omega, \sF) \times (\calD_\infty, \calB(\calD_\infty))$ 
or  $(\Omega, \sF) \times (\calD_T, \calB(\calD_T))$. 
We also define the averaged or annealed measure $\bfP$ on 
$(\calD_\infty, \calB(\calD_\infty))$ or  $(\calD_T, \calB(\calD_T))$ by
\be \label{e:bfPdef}
 \bfP(G) = \bE P^0_\om(G). 
\ee

\begin{definition}\label{j1.2}
For a bounded function $F$ on $\calD_T$ and a constant matrix $\Sigma$, let 
$\Psi^F_\eps = \qE^0_\omega F(X^{\eps})$ and 
$\Psi^F_\Sigma = \EBM F(\Sigma W)$. We will use $I$ to denote the identity matrix.

\sm (i) We say that the {\em Quenched Functional CLT} (QFCLT) holds 
for $X$ with limit $\Sigma W$ if for every $T>0$ and 
every bounded continuous function $F$ on $\calD_T$ we 
have $\Psi^F_\eps \to \Psi^F_\Sigma$ as $\eps\to 0$, with $\Pp$-probability 1.\\
(ii) We say that the {\em Weak Functional CLT} (WFCLT) 
holds for $X$ with limit $\Sigma W$ if for every $T>0$ and every 
bounded continuous function $F$ on $\calD_T$ we have 
$\Psi^F_\eps \to \Psi^F_\Sigma$ as $\eps\to 0$, in $\Pp$-probability.\\
(iii) We say that the {\em Averaged (or Annealed) Functional CLT}
(AFCLT) holds for $X$ with limit $\Sigma W$ if for every $T>0$ and every 
bounded continuous function $F$ on $\calD_T$ we have 
$ \bE \Psi^F_\eps \to \Psi_{\Sigma}^F$.
This is the same as standard weak convergence with respect to the probability measure $\bfP$. 
\end{definition}

If we take $\Sigma$ to be non-random then, since $F$ is bounded, it is
immediate that QFCLT $\Rightarrow$ WFCLT. In general for the QFCLT the matrix
$\Sigma$ might depend on the environment $\mu_\cdot(\om)$. However, if
the environment is stationary and ergodic, then $\Sigma$ is a shift invariant
function of the environment, so must be $\bP$--a.s. constant.
In \cite{DFGW} it is proved that if $\mu_e$ is a stationary ergodic 
environment with $\bE \mu_e<\infty$ then the WFCLT holds. In \cite[Theorem 1.3]{BBT1} 
it is proved that for the random conductance model the AFCLT and WFCLT are equivalent.

\begin{definition}
We say an environment $(\mu_e)$ on $\bZ^d$ is {\em symmetric} if the law of  $(\mu_e)$ is 
invariant under symmetries of $\bZ^d$. 
\end{definition}

If $(\mu_e)$ is stationary, ergodic and symmetric, and the WFCLT holds with
limit $\Sigma W$ then the limiting covariance matrix $\Sigma^T \Sigma$ must also
be invariant under symmetries of $\bZ^d$, so must be a constant 
times the identity.

In a previous paper \cite{BBT1} we proved the following theorem:

\begin{theorem}\label{T:oldmain}
Let $d=2$ and $p<1$.
There exists a symmetric stationary ergodic environment $\{\mu_e\}_{e\in E_2}$
with $\bE (\mu_e^p \vee \mu_e^{-p})<\infty$ 
and a sequence $\eps_n \to 0$ such that\\
(a)  the WFCLT holds for $X^{\eps_n}$ with limit $W$, 
i.e., for every $T>0$ and every 
bounded continuous function $F$ on $\calD_T$ we have 
$\Psi^F_{\eps_n} \to \Psi^F_I$ as $n\to \infty$, in $\Pp$-probability,
\\
but \\
(b) the QFCLT does not hold for  $X^{\eps_n}$ with limit $ \Sigma W$ for any $\Sigma$.  
\end{theorem}

In this paper we prove that for an environment similar to
that in Theorem \ref{T:oldmain} the WFCLT holds for $X^{\eps}$ as $\eps \to 0$,
and not just along a subsequence.

\begin{theorem}\label{T:main}
Let $d=2$ and $p<1$.
There exists a symmetric stationary ergodic environment $\{\mu_e\}_{e\in E_2}$
with $\bE (\mu_e^p \vee \mu_e^{-p})<\infty$ 
such that\\
(a)  the WFCLT holds for $X^{\eps}$ with limit $W$, 
i.e., for every $T>0$ and every 
bounded continuous function $F$ on $\calD_T$ we have 
$\Psi^F_{\eps} \to \Psi^F_I$ as $\eps \to 0$, in $\Pp$-probability,
\\
but \\
(b) the QFCLT does not hold for  $X^{\eps}$ with limit $ \Sigma W$ for any $\Sigma$.  
\end{theorem}

For more remarks on this problem see \cite{BBT1}.

\sm {\bf Acknowledgment.}
We are grateful to Emmanuel Rio, Pierre Mathieu, Jean-Dominique Deuschel 
and Marek Biskup for some very useful discussions.

\section{Description of the environment}\label{const}  

Here we recall the environment given in \cite{BBT1}. We refer the reader to that
paper for proofs of some basic properties.

Let $\Omega = (0,\infty)^{E_2}$, and $\calF$ be the Borel $\sigma$-algebra defined 
using the usual product topology. Then every $t\in\Z^2$ defines a transformation 
$T_t (\omega)=\omega +t$ of $\Omega$. Stationarity and ergodicity of the measures 
defined below will be understood with respect to these transformations. 

All constants (often denoted $c_1, c_2$, etc.) are assumed to be strictly positive and finite.
For a set $A \subset \bZ^2$ let $E(A)\subset E_2$ be the set of all edges with both endpoints in
$A$. Let $E_h(A)$ and $E_v(A)$ respectively
be the set of horizontal and vertical edges in $E(A)$.
Write $x \sim y$ if $\{x,y\}$ is an edge in $\bZ^2$. Define the exterior boundary of $A$ by
$$ \pd A =\{ y \in \bZ^2 -A: y \sim x \text{ for some } x \in A \}. $$
Let also
$$ \pd_i A = \pd(\bZ^2 -A). $$ 
Define balls in the $\ell^\infty$ norm by $\ball(x,r)= \{y: ||x-y||_\infty \le r\}$; of 
course this is just the square with center $x$ and side $2r$.

Let $\{a_n\}_{n\geq 0}$,  $\{ \beta_n\}_{n \ge 1}$ and $\{b_n\}_{n\geq 1}$ be  
strictly increasing sequences of positive integers growing to infinity with $n$,
with 
$$ 1=a_0 < b_1 < \beta_1 < a_1 \ll b_2 <  \beta_2<   a_2 \ll b_3 \dots $$
We will impose a number of conditions on these sequences in the course
of the paper. We collect the main ones here.
There is some redundancy in the conditions, for easy reference.

\begin{enumerate}[(i)]
\item  $a_n$ is even  for all $n$. 
\item For each $n \ge 1$, $a_{n-1}$ divides $b_n$, 
and $b_n$ divides $\beta_n$ and $a_n$. 
\item  $b_1 \geq 10^{10}$.
\item  $a_n/\sqrt{2n} \le     b_n \le a_n /  \sqrt{n} $ for all $n$, and
$b_n \sim a_n/\sqrt{n}$.
\item $b_{n+1} \ge 2^n b_n$ for all $n$.
\item $b_n > 40 a_{n-1}$ for all $n$.
\item $b_n$ is large enough so that the estimates (5.1) and (6.1) of \cite{BBT1} hold.
\item $100 b_n  < \beta_n \le  b_n n^{1/4}  < 2 \beta_n <   a_n/10$ for $n$ large enough.
\end{enumerate}

In addition, at various points in the proof, we will assume that $a_n$ is sufficiently much
larger than $b_{n-1}$ so that a process $X^{(n-1)}$ defined below is such that for $a\ge a_n$
the rescaled process
$$ (a^{-1} X^{(n-1)}_{a^2 t}, t\ge 0)$$
is sufficiently close to Brownian motion.
We will mark the places in the proof where we impose these extra conditions by \assmp.

\smallskip\noindent
We begin our construction by defining a collection of squares in $\bZ^2$. Let
\begin{align*} 
B_n &= [0, a_n]^2,  \\
B_n' &= [0, a_n-1]^2 \cap \bZ^2 ,\\
\calS_n(x) &= \{ x + a_n y + B_n': \, y \in \bZ^2 \}.
\end{align*} 
Thus $\sS_n(x)$ gives a tiling of $\bZ^2$ by disjoint squares of side $a_n-1$
and period $a_n$.
We say that the tiling $\calS_{n-1}(x_{n-1})$ is a refinement
of $\calS_n(x_n)$ if every square $Q \in \calS_n(x_n)$ is a finite
union of squares in $\calS_{n-1}(x_{n-1})$. It is clear that 
$\calS_{n-1}(x_{n-1})$ is a refinement of $\calS_n(x_n)$ if
and only if $x_n = x_{n-1}+ a_{n-1}y$ for some $y \in \bZ^2$.

Take $\sO_1$ uniform in $B'_1$, and for $n\geq 2$
take $\sO_n$, conditional on $(\sO_1, \dots, \sO_{n-1})$, 
to be uniform in $B'_n \cap ( \sO_{n-1} + a_{n-1}\bZ^2)$. We now define random tilings by letting
\bes
 \sS_n = \sS_n( \sO_n), \, n \ge 1.  
\ees

Let $\eta_n$, $K_n$ be positive constants; we will have $\eta_n \ll 1 \ll K_n$.
We define conductances  on $E_2$ as follows. 
Recall that $a_n$ is even, and let $a_n' = \frac12 a_n$. Let
$$ C_n = \{ (x,y) \in B_n \cap \bZ^2: y \ge x, x+y \le a_n \}. $$
We first define conductances $\nu^{n,0}_e$ for $e \in E(C_n)$. Let
\begin{align*}
D_n^{00} &= \big\{ (a'_n - \beta_n,y), a'_n - 10 b_n \le y \le a'_n + 10 b_n \big\},  \\
D_n^{01} &= \big\{ (x, a'_n + 10 b_n),  (x, a'_n + 10 b_n + 1), (x, a'_n - 10 b_n),  (x, a'_n - 10 b_n -1),   \\
\nonumber   
   & \q \q \q a'_n -\beta_n -b_n \le x \le a'_n -\beta_n + b_n \big\}.
\end{align*}
Thus the set $D^{00}_n \cup D_n^{01}$ resembles the letter I (see Fig.~\ref{fig1}).

For an edge $e \in E(C_n)$ we set 
\begin{align*} 
 \nu^{n,0}_{e}  &= \eta_n \q \text {if } e \in E_v(D^{01}_n), \\
 \nu^{n,0}_{e}  &= K_n \q \text {if } e \in E(D^{00}_n), \\
  \nu^{n,0}_{e}  &= 1 \q \text {otherwise.} 
\end{align*} 

\begin{figure} \includegraphics[width=4cm]{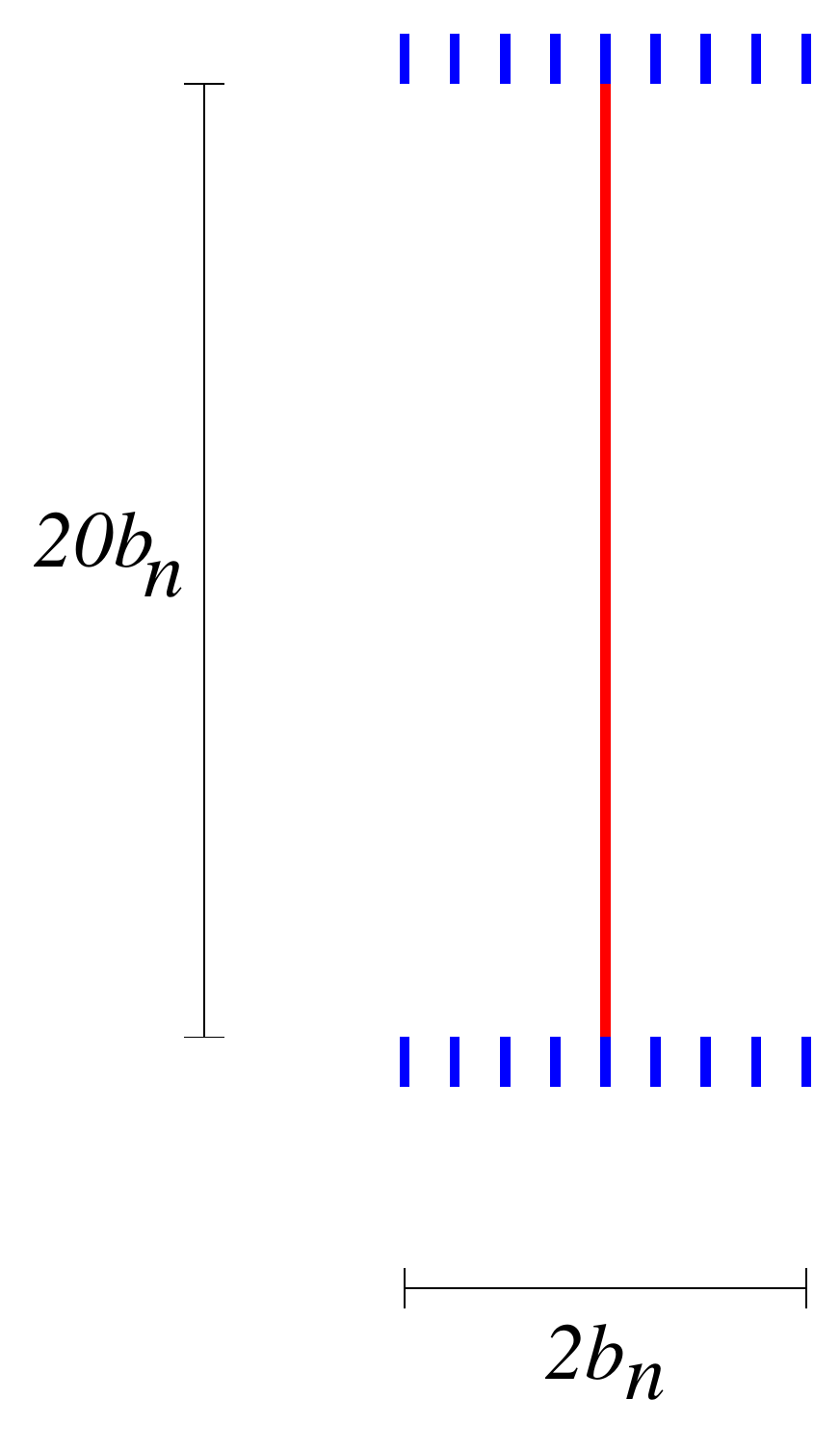}
\caption{The set $D^{00}_n \cup D_n^{01}$ resembles the letter I.
Blue edges have very low conductance. The red line represents edges with very 
high conductance. Drawing not to scale. 
}
\label{fig1}
\end{figure}

We then extend $\nu^{n,0}$ by symmetry to $E(B_n)$.
More precisely,
for $z =(x,y) \in B_n$, let $R_1 z=( y,x)$ and  $R_2z = (a_n-y,a_n-x)$, so that
$R_1$ and $R_2$ are reflections in the lines $y=x$ and $x+y=a_n$.
We define $R_i$ on edges by $R_i (\{x,y\}) = \{R_i x, R_i y \}$ for $x,y \in B_n$. 
We then extend $\nu^{0,n}$ to $E( B_n)$ so that
$\nu^{0,n}_e =  \nu^{0,n}_{R_1 e }=\nu^{0,n}_{R_2 e }$ for  $e \in E(B_n)$.
We define the {\em obstacle} set $D_n^0$ by setting
$$ D_n^0 = \bigcup_{i=0}^1 \big( D_n^{0,i} \cup R_1(D_n^{0,i})  \cup  R_2(D_n^{0,i})
 \cup R_1R_2 (D_n^{0,i} ) \big). $$
Note that $\nu^{n,0}_e=1$ for every edge adjacent to the boundary of $B_n$,
or indeed within a distance $ a_n/4$ of this boundary.
If $e=(x,y)$, we will write $e-z = (x-z,y-z)$. 
Next we extend $\nu^{n,0}$ to $E_2$ by periodicity, i.e.,
$\nu^{n,0}_e = \nu^{n,0}_{e+ a_n x}$ for all $x\in \Z^2$.
We define the conductances $\nu^n$ by translation by $\sO_n$, so that
\bes
 \nu^n_e =\nu^{n,0}_{e-\sO_n}, \, e \in E_2.
\ees
We also define the obstacle set at scale $n$ by
\be\label{ma26.1}
 D_n = \bigcup_{ x \in \bZ^2} (a_n x + \sO_n + D^0_n ).
\ee
We will sometimes call the set $D_n$ the set of $n$th level obstacles.

We define the environment $\mu^n_e$ inductively by
\begin{align*}
 \mu^n_e &= \nu^{n}_e \q \text{ if } \nu^n_e \neq 1, \\
 \mu^n_e &= \mu^{n-1}_e \q \text{ if } \nu^n_e=1.
\end{align*}
Once we have proved the limit exists, we will set
\be \label{e:mudef}
 \mu_e = \lim_n \mu^n_e.
\ee

\begin{lemma} \label{L:erg} (See  \cite[Theorem 3.1]{BBT1}).\\
(a) The environments $(\nu^n_e, e\in E_2)$, $(\mu^n_e, e\in E_2)$
are stationary, symmetric and ergodic.\\
(b) The limit \eqref{e:mudef} exists $\bP$--a.s. \\
(c) The environment $(\mu_e, e \in E_2)$ is  stationary, symmetric and ergodic.
\end{lemma}

Now let 
\begin{align}\label{j27.4}
\calL_n f(x) = \sum_{y} \mu^n_{xy} (f(y)-f(x)), 
\end{align}
and $X^{(n)}$ be the associated Markov process. Set
\be \label{e:etadef}
  \eta_n =   b_n^{-(1+1/n)}, \, n \ge 1.
\ee
From Section 4 of \cite{BBT1} we have:

\begin{theorem} \label{T:eK}
For each $n$ there exists a constant $K_n$, depending on $\eta_1, K_1, \dots \eta_{n-1}, K_{n-1}$,
such that the QFCLT holds for $X^{(n)}$ with limit $W$.
\end{theorem}

For each $n$ the process $X^{(n)}$ has invariant measure which is counting measure 
on $\bZ^2$. For $x \in \bR^2$ and $a>0$ write $[xa]$ for the point in $\bZ^2$ closest to $xa$.
(We use some procedure to break ties.) We have the following bounds on the transition
probabilities of $X^{(n)}$ from \cite{BZ}. We remark that the constant $M_n$ below is
not effective -- i.e. the proof does not give any control on its value. 
Write $k_t(x,y) = (2\pi t)^{-1} \exp( -|x-y|^2/2t)$ for the transition density of Brownian motion
in $\bR^2$, and
$$ p^{\om,n}_t(x,y) = P^x_\om( X^{(n)}_t =y )$$
for the transition probabilities for $X^{(n)}$.

\begin{lemma} \label{L:hkXn} 
For each $0< \delta < T$  there exists $M_n=M_n(\delta,T)$ such that for $a \ge M_n$ 
\be \label{e:GB1}
\half k_t(x,y) \le a^{2} p^{\om,n}_{a^2t}([xa],[ya])  \le 2 k_t(x,y) \, \hbox { for all }
 \delta \le t \le T, |x|, |y| \le T^2.
\ee
\end{lemma}

\section{Preliminary results}

Since a proof of Theorem \ref{T:oldmain}(b) was given in \cite{BBT1}, 
all we need to prove is part (a) of Theorem \ref{T:main}.
The argument consists of several lemmas. We start with some preliminary 
results on weak convergence of probability measures on the space of c\`adl\`ag functions. 
Recall the definitions of the measures $\bP$ and $P^0_\om$.

Recall that  $\calD := \calD_1 = D([0,1], \R^2)$ denotes the space of c\`adl\`ag functions 
equipped with the Skorokhod metric $\Sd$ defined as follows (see \cite[p.~111]{B}). 
Let $\Lambda$ be the family of continuous strictly increasing functions $\lambda$ 
mapping $[0,1]$ onto itself. In particular, $\lambda(0) =0$ and $\lambda(1) =1$. 
If $x(t), y(t) \in  \calD$ then  
\begin{align*}
\Sd(x,y) = \inf_{\lambda \in \Lambda}
\max\Big( \sup_{t\in[0,1]} |\lambda(t) - t|, \sup_{t\in[0,1]} |y(\lambda(t)) - x(t)| \Big).
\end{align*}
For $x(t) \in  \calD$, let $\Osc(x, \delta) = \sup\{|x(t)-x(s)|: s,t\in[0,1], |s-t|\le \delta\}$.

\begin{lemma}\label{d21.2}
Suppose that $\sigma: [0,1] \to [0,1]$ is continuous, non-decreasing and $\sigma(0) = 0$ 
(we do not require that $\sigma(1) = 1$).  
Suppose that $|\sigma(t) - t| \le \delta$ for all $t\in[0,1]$.
Let  $\eps\geq0$, $\delta_1>0$, $x, y \in  \calD$ with
$\Sd(x(\,\cdot\,), y(\,\cdot\,))\le \eps$, and
$\Osc(x, \delta) \vee \Osc(y, \delta) \le \delta_1$. Then
$\Sd(x(\sigma(\,\cdot\,)), y(\sigma(\,\cdot\,))) \le \eps + 2\delta_1$.
\end{lemma}

\begin{proof}
For any $\eps_1> \eps$ there exists $\lambda\in \Lambda$ such that,
\begin{align*}
\max\Big( \sup_{t\in[0,1]} |\lambda(t) - t|,
\sup_{t\in[0,1]} |y(\lambda(t)) - x(t)| \Big)\le\eps_1.
\end{align*}
We have for $\lambda$ satisfying the above condition,
\begin{align*}
&\sup_{t\in[0,1]} |y(\sigma(\lambda(t))) - x(\sigma(t))|\\
&\qquad \le 
\sup_{t\in[0,1]} (|y(\sigma(\lambda(t))) - y(\lambda(t))|
+ |y(\lambda(t)) - x(t)| + |x(t) - x(\sigma(t))|) \\
&\qquad \le \Osc(y,\delta) + \eps_1 + \Osc(x,\delta) \le \eps_1 + 2 \delta_1. 
\end{align*}
Hence,
\begin{align*}
\max\Big( \sup_{t\in[0,1]} |\lambda(t) - t|,
\sup_{t\in[0,1]} |y(\sigma(\lambda(t))) - x(\sigma(t))|
\Big) \le \eps_1 + 2 \delta_1.
\end{align*}
Taking infimum over all $\eps_1 > \eps$ we obtain
$\Sd(x(\sigma(\,\cdot\,)), y(\sigma(\,\cdot\,)))
\le \eps + 2\delta_1$.
\end{proof}

Let $\dP$ denote the Prokhorov distance between probability measures on a probability space defined 
as follows (see \cite[p.~238]{B}). 
Recall that 
$\Omega = (0,\infty)^{E_2}$ and $\calF$ is the Borel $\sigma$-algebra defined 
using the usual product topology.
We will use measurable spaces $(\calD_T, \calB(\calD_T))$ and 
 $(\Omega, \sF) \times (\calD_T, \calB(\calD_T))$, for a fixed $T$ (often $T=1$).
Note that $\calD_T$ and $\Omega \times \calD_T$ are metrizable, with the metrics generating the usual topologies. A ball around a set $A$ with radius $\eps$ will
be denoted $\calB(A,\eps)$ in either space. 
For probability measures $P$ and $Q$, 
$\dP(P,Q) $ is the infimum of $\eps>0$ such that $P(A) \le Q(\calB(A,\eps)) + \eps$ and 
$Q(A) \le P(\calB(A,\eps)) + \eps$ for all Borel sets $A$.
Convergence in the metric $\dP$ is equivalent to the weak convergence of measures.
By abuse of notation we will sometimes write arguments of the function 
$\dP(\,\cdot\,,\,\cdot\,)$ as processes rather than their distributions:  for example we will write
$\dP( \{(1/a)X^{(n)}_{ta^2}, t\in[ 0,1]\}, \PBM)$.
We will use $\dP$ for the Prokhorov distance
between probability measures on $(\Omega, \sF) \times (\calD_T, \calB(\calD_T))$. We will write $\dP_\omega$ for the metric on the space
$(\calD_T, \calB(\calD_T))$.
It is straightforward to verify that if,  for some processes $Y$ and $Z$, 
$\dP_\omega(Y,Z) \le \eps$ for $\bP$--a.a. $\omega$, then $\dP(Y,Z) \le \eps$.

We will sometimes write $W(t)=W_t$ and similarly for other processes.

\begin{lemma}\label{d21.1}
There exists a function $\rho: (0,\infty) \to (0,\infty)$ such that $\lim_{\delta\downarrow 0}
 \rho(\delta) = 0$ and  the following holds.
Suppose that $\delta,\delta'\in (0,1)$ and $\sigma: [0,1] \to [0,1]$ is a non-decreasing 
stochastic process such that $t-\sigma_t \in [0,\delta]$ for all $t$, with probability greater 
than $1-\delta'$. Suppose that $\{W_t, t\geq 0\}$ has the distribution $\PBM$ and 
$W^*_t = W(\sigma_t)$ for $t\in[0,1]$. 
Then $\dP(\{W^*_t, t\in[0,1]\}, \PBM) \le \rho(\delta) + \delta'$.
\end{lemma}

\begin{proof}
Suppose that $W, W^*$ and $\sigma$ are defined on the sample space with a 
probability measure $P$.
It is 
easy to see that we can choose $\rho(\delta)$ so that 
$\lim_{\delta\downarrow 0} \rho(\delta) = 0$
and $P(\Osc(W,\delta) \geq\rho(\delta) )<\rho(\delta)$. 
Suppose that the event 
$F := \{\Osc(W,\delta) <\rho(\delta)\}\cap \{ \forall t\in[0,1]: t-\sigma_t \in [0,\delta]\}$ holds. 
Then taking $\lambda(t) = t $,
\begin{align*}
\Sd(W, W^*) &\le \max\Big( \sup_{t\in[0,1]} |\lambda(t) - t|,
\sup_{t\in[0,1]} |W(\lambda(t)) - W^*(t)| \Big) \\
&= \sup_{t\in[0,1]} |W(t) - W(\sigma(t))| \le \Osc(W, \delta) < \rho(\delta).
\end{align*}
We see that if $F$ holds and $W \in A \subset \calD$ then 
$W^*(\,\cdot\,)\in \calB( A,\rho(\delta))$.
Since $P(F^c) \le \rho(\delta) + \delta'$, we obtain
\begin{align*}
P&(W \in A) \\
&\leq P(\{W\in A\} \cap F) + P(F^c) 
\leq P(\{W^*\in \calB( A,\rho(\delta))\} \cap F) 
+\rho(\delta) + \delta'\\
&\leq P(W^*\in \calB( A,\rho(\delta))) 
+\rho(\delta) + \delta'.
\end{align*}
Similarly we have
$P(W^*\in A ) \le P(W\in \calB( A,\rho(\delta)) ) + \rho(\delta) + \delta'$, and
the lemma follows.
\end{proof}

\begin{lemma}\label{ma26.5}
Suppose that for some processes $X, Y$ and $Z$ on the interval $[0,1]$ we have $Z= X+Y$ and $P(\sup_{0\leq t \leq 1} |X_t| \leq \delta) \geq 1-\delta$. 
Then $\dP(\{Z_t, t\in[0,1]\}, \{Y_t, t\in[0,1]\}) \le \delta$.
\end{lemma}

\begin{proof}
Suppose that the event 
$F := \{\sup_{0\leq t \leq 1} |X_t| \leq \delta\}$ holds. 
Then taking $\lambda(t) = t $,
\begin{align*}
\Sd(Z,Y) &\le \max\Big( \sup_{t\in[0,1]} |\lambda(t) - t|,
\sup_{t\in[0,1]} |Z(\lambda(t)) - Y(t)| \Big) \\
&= \sup_{t\in[0,1]} |Z(t) - Y(t)| \le \delta.
\end{align*}
We see that if $F$ holds and $Z \in A \subset \calD$ then 
$Y(\,\cdot\,)\in \calB( A,\delta)$.
Since $P(F^c) \le \delta$, we obtain
\begin{align*}
P(Z \in A) &\leq P(\{Z\in A\} \cap F) + P(F^c) 
\leq P(\{Y\in \calB( A,\delta)\} \cap F) 
+ \delta\\
&\leq P(Y\in \calB( A,\delta)) 
+ \delta.
\end{align*}
Similarly we have
$P(Y\in A ) \le P(Z\in \calB( A,\delta) ) +  \delta$, and
the lemma follows.
\end{proof}

Recall that the function $e\to \mu^n_e$ is periodic with period $a_n$.
Hence the random field $\{\mu^n_e\}_{e\in E_2}$ takes only finitely many values --
this is a much stronger statement than the fact that $\mu^n_e$
takes only finitely many values.

By Theorem \ref{T:eK} for each $n \ge 1$,
$$ \lim_{a\to \infty} 
\dP( \{  (1/a)X^{(n)}_{ta^2}, t\in[ 0,1]\}, \PBM) =0. $$ 
Thus \assmp we can take $a_{n+1}$ so large that for every $\omega$, $n \ge 1$
and $a\geq a_{n+1}$, 
\be \label{e:PdistBM} 
\dP_\omega( \{(1/a)X^{(n)}_{ta^2}, t\in[ 0,1]\}, \PBM) \le 2^{-n}.
\ee

\bigskip

Let $\theta$ denote the usual shift operator for Markov processes, that is, $X^{(n)}_t \circ \theta_s = X^{(n)}_{t+s}$ for all $s,t\geq 0$ (we can and do assume that $X^{(n)}$ is the canonical process on an appropriate probability space). 
Recall that
$\ball(x,r) =\{y: ||x-y||_\infty \le r\}$ denote balls in the $\ell^\infty$ norm
in $\bZ^2$ (i.e. squares), $a_n' = a_n/2$, $B_n=[0,a_n]^2$ and $u_n =(a_n', a'_n)$. Note that $u_n$ is 
the center of $B_n$.
We choose $ \beta_n$  so that
\begin{align}\label{j2.10}
 b_n n^{1/8}
< \beta_n \leq \lfloor b_n n^{1/4}\rfloor  < 2 \beta_n  < a_n/10,
\end{align}
and we assume that $n$ is large enough so that the above inequalities hold.
Let $\sC_n =\{ u_n + \sO_n+ a_n \bZ^2\}$ be the set of centers of the squares in $\sS_n$, and let
\be \label{e:Hdef}
  \sK(r) = \bigcup_{z \in \sC_n} \ball(z,r).
\ee
Now let
\begin{align*}
\Gamma^1_n &= \sK(2\beta_n), \\
\Gamma^2_n &= \Z^2 \setminus \sK(4 \beta_n).
\end{align*}
Now define stopping times as follows.
\begin{align*}
S^n_0 &= T^n_0 = 0,\\
U^n_k & = \inf\{t\geq S^n_{k-1}: X^{(n)}_t \in \Gamma^2_n\}, \qquad k \geq 1,\\
S^n_k & = \inf\{t \geq U^n_k: X^{(n)}_t \in \Gamma^1_n\}, \qquad k \geq 1,\\
V^n_1 & = \inf \Big\{t \in \bigcup_{k\geq 1} [U^n_k, S^n_k]: 
X^{(n)}_t \in X^{(n)}(T^n_0) + a_{n-1} \bZ^2 \Big\}, \\
T^n_k & = \inf\{t\geq V^n_{k}: X^{(n)}_t \in \Gamma^1_n\}, \qquad k \geq 1,\\
V^n_k & = V^n_{1} \circ \theta_{T^n_{k-1}} , \qquad k \geq 2.
\end{align*}
Let
$$ J= \bigcup_{k=1}^\infty [V^n_k, T^n_k]; $$
for $t \in J$  the process $X^{(n)}$ is a distance at least
$\beta_n$ away from any $n$th level obstacle.
Now set for $t\geq 0$,
\begin{align*}
\sigma^{n,1}_t &= \int_0^t \bone_J(s) ds = \sum_{k=1}^\infty \left(T^n_{k} \land t - V^n_{k} \land t\right),\\
\sigma^{n,2}_t &= t-\sigma^{n,1}_t = \sum_{k=0}^\infty \left(V^n_{k+1} \land t - T^n_{k} \land t\right).
\end{align*}
Let $\wh \sigma^{n,j}$ denote the right continuous inverses of these processes, given by
\bes
\wh \sigma^{n,j}_t = \inf\{s\geq 0: \sigma^{n,j}_s \geq t\}, \, j=1,2.
\ees
Finally let 
\begin{align*}
X^{n,1}_t &= X^{(n)}_0 + \int_0^t \bone_J(s) dX^{(n)}_s  \\
 &=  X^{(n)}_0 + \sum_{k=0}^\infty
\left(X^{(n)}(T^n_k \land t) - X^{(n)}(V^n_{k} \land t)\right), \\
\wh X^{n,1}_t &=  X^{(n)}_0 + X^{n,1}(\wh \sigma^{n,1}_t), \\
X^{n,2}_t &=  X^{(n)}_0 + \int_0^t \bone_{J^c}(s) dX^{(n)}_s  \\
 &=  X^{(n)}_0 + \sum_{k=0}^\infty
\left(X^{(n)}(V^n_{k+1} \land t) - X^{(n)}(T^n_{k} \land t)\right), \\
\wh X^{n,2}_t &=  X^{(n)}_0 + X^{n,2}(\wh \sigma^{n,2}_t).
\end{align*}

The point of this construction is the following.
For every fixed $\omega$, the function $e\to\mu_e^{n-1}$ is invariant under the shift by 
$x a_{n-1}$ for any $x\in \Z^2$, and 
$X^{(n)}(V^n_{k+1} ) = X^{(n)}(T^n_{k}) + x a_{n-1} $ for some $x\in \Z^2$. 
It follows that for each $\omega\in \Omega$, we have the following equality of  distributions: 
\be \label{e:Xhatdsn}
\{\wh X^{n,1}_t, t\geq 0\} \eqd \{X^{(n-1)}_t, t\geq 0\}.
\ee
The basic idea of the argument which follows is to write $X^{(n)}= X^{n,1} + X^{n.2}$.
By Theorem \ref{T:eK}, or more precisely by \eqref{e:PdistBM}, the process $X^{n,1}$ is close
to Brownian motion, so to prove Theorem \ref{T:main} we need to prove that $X^{n,2}$
is small.

\bigskip
We state the next lemma at a level of generality greater than what we need in this article. A variant of our lemma is in the book \cite{AF} but we could not find a statement that would match perfectly our needs.
Consider a finite graph $G=(\calV,E)$ and suppose that for any edge $\ol{xy}$, $\mu_{xy}$ is a non-negative real number. Assume that $\sum_{y\sim x} \mu_{xy} >0$ for all $x$.
For $f: \calV \to \bR$ set
$$ \sE (f,f) = \sum_{\{x,y\} \in E}  \mu_{xy} (f(y)-f(x))^2. $$
Suppose that $A_1, A_2 \subset \calV$, $A_1 \cap A_2 = \emptyset$, and
let 
\begin{align*}
\sH &=\{ f:\calV \to \bR \text{ such that } f(x)=0 \text{ for } x\in A_1, f(y)=1 \text{  for } y \in A_2\},\\
 \res^{-1} &= \inf\{ \sE(f,f): f \in \sH \}. 
\end{align*}
Thus $\res$ is the effective resistance between $A_1$ and $A_2$.
Let $Z$ be the continuous time Markov process on $\calV$ with the generator $\calL$  given by
\begin{align}\label{ma21.1}
\calL f(x) = \sum_y \mu_{xy} (f(y) - f(x)).
\end{align}
Let $T_i = \inf\{t\geq 0: Z_t \in A_i\}$  for $i=1,2$,
 and let $Z^{(i)}$ be $Z$ killed at time $T_i$.

\begin{lemma}\label{L:com}
There exist probability measures $\nu_1$ on $A_1$ and $\nu_2$ on $A_2$ such that
\begin{align*}
E^{\nu_2} T_1 + E^{\nu_1} T_2 = \res |\calV| . 
\end{align*}
Moreover, for $i=1,2$, $\nu_i$ is the capacitary measure of $A_i$ for the process $Z^{(3-i)}$.
\end{lemma}

\begin{proof}
Let $h_{12}(x) = P^x( T_1 < T_2)$. 
Set $D= \calV-A_1$ and recall that $Z^{(i)}$ is $Z$ killed at time $T_i$.
Let $G_2$ be the Green operator for $Z^{(2)}$, and $g_2(x,y)$ be the density of
$G_2$ with respect to counting measure, so that
$$ E^x T_2 = \sum_{y \in \calV} g_2(x,y). $$
Note that $g_2(x,y)=g_2(y,x)$.
Let $e_{12}$ be the capacitary measure of $A_1$ for the process $Z^{(2)}$. Then
$\res^{-1} = \sum_{z \in A_1} e_{12}(z), $ and
$$ h_{12}(x) =  \sum_{z \in A_1} e_{12}(z) g_2(z,x) . $$
So, if $ \nu_1 = \res e_{12} $, then
\begin{align*}
\sum_{y \in \calV} h_{12}(y) &= 
   \sum_{y \in \calV}  \sum_{x \in A_1} e_{12}(x) g_2(x,y)  \\
&= \res^{-1} \sum_{x \in A_1} \nu_1(x)  \sum_{y \in \calV} g_2(x,y)  \\
&= \res^{-1} \sum_{x \in A_1} \nu_1(x) E^x T_1 = \res^{-1} E^{\nu_1} T_2.
\end{align*}
Similarly if $h_{21}(x) =\bP^x( T_2 < T_1)$ we obtain
$\res^{-1} E^{\nu_2} T_1 = \sum_{y \in \calV} h_{21}(y) $, and since $h_{12}+h_{21}=1$, adding these
equalities proves the lemma.
\end{proof}

\section{ Estimates on the process $X^{n,2}$ } 

In this section we  will prove 

\begin{proposition}\label{d22.2}
For every $\delta>0$ there exists $n_1$ such that for all $n\geq n_1$,  $u\geq a_n^2$, and $\omega$ such that $0 \notin \Gamma^1_n \setminus \prt_i \Gamma^1_n$, 
\begin{align}\label{d22.3}
P^0_\omega
\left(  \sigma^{n,2}_u / u \le \delta, \sup_{0\le s \le u} u^{-1/2} |X^{n,2}_s| \le \delta \right) \geq 1-\delta.
\end{align}
\end{proposition}

The proof requires a number of steps. We begin with a Harnack inequality.

\begin{lemma}\label{L:harn}
Let $1 \le \lam \le 10$. 
There exist $p_1>0$ 
 and $n_1 \ge 1$ with the following properties. \\
(a) Let $x \in \bZ^2$, let  $B_1= \ball(x, \lam \beta_n)$ 
and $B_2= \ball(x, (2/3) \lam \beta_n)$. 
Let $F$ be  the event that $X^{(n)}$ makes a closed loop around $B_2$
inside $B_1 - B_2$
before its first exit from $B_1$.
If $n \ge n_1$ and $D_n \cap B_1 = \emptyset$ then
$P^y_\om(F) \ge p_1$ for all $ y \in B_2$. \\
(b) Let $h$ be harmonic in $B_1$.  
Then 
\be \label{e:harni}
  \max_{B_2} h \le p_1^{-1} \min_{B_2} h.
\ee
\end{lemma}

\begin{proof}
(a) Using \assmp and \eqref{e:PdistBM} we can make a Brownian approximation to
$\beta_n^{-1} X^{(n)}_\cdot$ which is good enough so that this estimate holds.\\
(b) Let $y \in B_1$ be such that $h(y) = \max_{z \in B_2} h(z)$.
Then by the maximum principle there exists a connected path $\gamma$ from $y$
to $\pd_i B_1$ with $h(w) \ge h(y)$ for all $w \in \gamma$.
Now let $y'\in B_2$.  On the event $F$ the process $X^{(n)}$ must hit $\gamma$, and
so we have
$$ h(y') \ge P^{y'}_\om (F) \min_{\gamma} h \ge p_1 h(y),$$
proving \eqref{e:harni}. 
\end{proof}

\begin{lemma}\label{12a29lem}
For some $n_1$ and $c_1$, for all $n\geq n_1$, $k\geq 1$, and $\omega$ such that 
$0 \notin \Gamma^1_n \setminus \prt_i \Gamma^1_n$, 
\begin{align}\label{z1.1}
E^0_\omega(U^n_k - S^n_{k-1}  \mid \calF_{S^n_{k-1}} ) \leq c_1 \beta_n^2.
\end{align}
\end{lemma}

\begin{proof}
Assume that $\omega$ is such that $0 \notin \Gamma^1_n \setminus \prt_i \Gamma^1_n$.
By the strong Markov property applied at $S^n_{k-1}$ for $k >1$, it is enough to prove 
the Lemma for $k=1$, that is that 
$E^x_\omega(U^n_1 ) \le c_1 \beta_n^2$ for all 
$x \notin \Gamma^1_n \setminus \prt_i \Gamma^1_n$.
 Let 
\begin{align*}
\calV &= \ball(u_n + \calO_n, 4 \beta_n+1),\\
A_1 &= \prt_i \ball(u_n + \calO_n, (3/2) \beta_n),\\
A_2 &= \prt_i \calV,\\
A_3 &= \prt_i \ball(u_n + \calO_n, 2 \beta_n)\\
T_i &= \inf\{t\geq 0: X^{(n)}_t \in A_i\}, \qquad i =1,2,3.
\end{align*}
Let $Z$ be the continuous time Markov chain defined on $\calV$ by \eqref{ma21.1}, 
relative to the environment $\mu^n$. Note that the transition probabilities from $x$ 
to one of its neighbors are the same for $Z$ and $X^{(n)}$ if $x $ is in the interior 
of $\calV$, i.e., $x\notin \prt_i \calV \cup(\Z^2 \setminus \calV)$. 
Note also that $Z$ and  $X^{(n-1)}$ have the same transition probabilities 
in the region between $A_1$ and $A_3$.
The expectations and  probabilities in this proof will refer to $Z$.
By Lemma \ref{L:com}, there exists a probability measure $\nu_1$ on $A_1$ such that
$E^{\nu_1} T_2 \leq \res |\calV|$.
We have $|\calV| \leq c_2 \beta_n^2$.

To estimate $\res$ note that by the choice of the constants $\eta_{n-1}$ and $K_{n-1}$
in Theorem \ref{T:eK}, the resistance (with respect to $\mu^{n-1}_e$) between two opposite sides of any
square in $\sS_{n-1}$ will be 1. It follows that the resistance 
between two opposite sides of any square side $\beta_n$ which is a union
of squares in $\sS_{n-1}$ will also be 1. So, using Thompson's principle as
in \cite{BB3} we deduce that $\res \leq c_3$.

So, by Lemma \ref{L:com} we have
\begin{align}\label{ma18.1}
E^{\nu_1} T_2 \leq c_4 \beta_n^2.
\end{align}

We have for some $c_5$, $p_1 >0$ all $n$ and $x\in \calV \setminus \ball(u_n + \calO_n, (3/2) \beta_n)$,
\begin{align*}
P^x_\omega ( T_1 \land T_2  \leq c_5 \beta_n^2) > p_1,
\end{align*}
because an analogous estimate holds for Brownian motion and \assmp we have  \eqref{e:PdistBM}. This and a standard argument based on the strong Markov property imply that for $x\in A_3$,
\begin{align*}
E^x_\omega ( T_1 \land T_2 ) \leq c_6 \beta_n^2.
\end{align*}

Now for $y \in A_1$ and $x \in \calV$ set 
$$ \nu_3^x (y) = P^x_\omega(X^{(n)}(T_1 \wedge T_2) = y  ).$$
(Note that there exist $x$ with $\sum_{y \in A_1} \nu^x_3(y) < 1$.)
We obtain for  $n\geq n_2$ and $x\in A_3$,
\begin{align}\label{ma21.2}
E^x_\omega ( T_2 ) &=
E^x_\omega ( T_1 \wedge T_2 ) + E^x_\omega ((T_2 -T_1) \bone_{ T_1 < T_2} )  \\
& = E^x_\omega ( T_1 \wedge T_2 )  + E^{\nu_3^x} T_2 
\leq  c_6 \beta_n^2 + E^{\nu_3^x}_\om T_2. \nonumber
\end{align}
For $y \in A_1$ the function $x \to \nu^x_3(y)$ is harmonic in $\calV \setminus A_1$.
So we can apply the Harnack inequality Lemma \ref{L:harn} to deduce that there exists
$c_7$ such that
\be
 \nu^x_3(y) \le c_7 \nu^{x'}_3(y) \hbox{ for all } x,x' \in A_3, y \in A_1.
\ee

The measure $\nu_1$ is  the hitting distribution on $A_1$ 
for the process $Z$ starting with $\nu_2$  (see \cite[Chap.~3, p.~45]{AF}). So for
any $x' \in A_3$,
\begin{align*}
\nu_1(y) &= P^{\nu_2}_0  ( Z_{T_1} =y) 
= \sum_{x \in A_3} P^{\nu_2}_0  ( Z_{T_1} =x ) P^x_\om( Z_{T_1} =y)   \\
&\ge \sum_{x \in A_3} P^{\nu_2}_0  ( Z_{T_1} =x ) P^x_\om( Z_{T_1 \wedge T_2} =y)  
\ge  \min_{x \in A_3}   \nu^x_3(y) \ge c_7^{-1}  \nu^{x'}_3(y).
\end{align*}
Hence for any $x \in A_3$,
$$  E^{\nu_3^x}_\om T_2 \le c_7 E^{\nu_1}_\om T_2 \le  c_8 \beta_n^2,$$
and combining this with \eqref{ma21.2} completes the proof. 
\end{proof}

Let
\begin{align*}
R_n^y =\inf\left\{t\geq 0 :  X^{(n)}_t \in (y + a_{n-1}\Z^2) \cup \Gamma^1_n\right\}.
\end{align*}

\begin{lemma}
There exist $c_1>0$ and $p_1<1$ such that for all $x,y \in \bZ^2$,
\begin{align}\label{ma23.1}
&P^x_\om\left(R_n^y \geq c_1 b_n^2 \right) \le p_1,\\
&P^x_\om\left(\sup_{0\leq t \leq R^y_n}|x- X^{(n)}_t| \geq c_1 b_n \right) \le p_1.
\label{ma23.2}
\end{align}
\end{lemma}

\begin{proof}
Recall that the family $\{\mu^{n-1}_{x+ \cdot}\}_{x\in \Z^2}$ of translates of the 
environment  $\mu^{n-1}_\cdot$ contains only a finite number of distinct elements.
Since each square in $\sS_{n-1}$ contains one point in $(y + a_{n-1}\Z^2)$, 
if $b_n/a_{n-1}$  is sufficiently large \assmp then using the transition density estimates \eqref{e:GB1}
as well as \eqref{e:PdistBM}, we obtain \eqref{ma23.1} and \eqref{ma23.2}. 
\end{proof}

\begin{lemma}
For some $n_1$ and $c_1$, for all $n\geq n_1$, $k\geq 1$, and $\omega$ such that $0 \notin \Gamma^1_n \setminus \prt_i \Gamma^1_n$, 
\begin{align}\label{ma21.5}
E^0_\omega(V^n_k - T^n_{k-1} \mid \calF_{T^n_{k-1}}) \leq c_1 b_n^2 n^{1/2}.
\end{align}
\end{lemma}

\begin{proof}
Assume that $\omega$ is such that $0 \notin \Gamma^1_n \setminus \prt_i \Gamma^1_n$.
Let
\begin{align*}
\wh R^n_k =\inf\left\{t\geq U^n_k :  X^{(n)}_t \in (X^{(n)}(T^n_0) + a_{n-1}\Z^2) \cup \Gamma^1_n\right\}.
\end{align*}
Let $F_k = \{ \wh R^n_k < S^n_k\}$ and $G_k = \bigcap _{j=1}^k F_j^c$.
Since $b_n n^{1/8} < \beta_n$ for large $n$, we obtain from \eqref{ma23.2} and definitions of $\Gamma^1_n, \Gamma^2_n, U^n_k$ and $S^n_k$ that there exists $p_2>0$ such that for $x\in \Gamma^2_n$,
\begin{align*}
P^x_\omega(F_k \mid \calF_{U^n_k}) > p_2.
\end{align*}
Hence,
\begin{align}\label{n5.2}
P^x_\omega(G_k ) < (1-p_2)^k.
\end{align}
Note that if $F_k$ occurs then $V^n_1 \leq \wh R^n_k$.
We have, using \eqref{z1.1}, \eqref{ma23.1} and \eqref{n5.2},
\begin{align*}
E^0_\omega(V^n_1 - T^n_{0} ) 
&\leq
\sum_{k=1}^\infty
E^0_\omega((U^n_k - S^n_{k-1}) \bone_{G_{k-1}})
+ 
\sum_{k=1}^\infty
E^0_\omega((\wh R^n_k - U^n_{k}) \bone_{G_{k-1}})\\
&\leq \sum_{k=1}^\infty
c_2 \beta_n^2 (1-p_2)^{k-1}
+ \sum_{k=1}^\infty c_3 b_n^2 (1-p_2)^{k-1}  \\
&\leq c_4 \beta_n^2 \leq c_5 b_n^2 n^{1/2}.
\end{align*}
This proves the lemma for $k=1$.
The general case is obtained by applying this estimate to the
process shifted by $T^n_{k-1}$; in other words, by using the strong Markov property.
\end{proof}

\begin{lemma} \label{L:sigma_1}
For every $\delta>0$ there exists $n_1$ such that for all $n\geq n_1$,  $u\geq a_n^2$, and $\omega$ such that $0 \notin \Gamma^1_n \setminus \prt_i \Gamma^1_n$,
\begin{align}\label{n8.5}
P^0_\omega
\left(  \sigma^{n,2}_u / u \le \delta \right) \geq 1-\delta/2.
\end{align}
\end{lemma}

\begin{proof}
Assume that $\omega$ is such that $0 \notin \Gamma^1_n \setminus \prt_i \Gamma^1_n$.
Fix an arbitrarily small $\delta >0$, consider $u \geq a_n^2$ and let $j_* = \lceil u/(b_n^2 n^{5/8} )\rceil$. Then
\eqref{ma21.5} implies that for some $c_1$ and $n_2$, all $n\geq n_2$, $u \geq a_n^2$, 
\begin{align*}
E^0_\om \left( \frac 1{j_*} \sum_{j=1}^{j_*} V^n_j - T^n_{j-1}\right) 
\leq c_1 b_n^2 n^{1/2}.
\end{align*}
Hence, for some $n_3$, all $n\geq n_3$, $u \geq a_n^2$,
\begin{align*}
P^0_\om \left( \frac 1{j_*} \sum_{j=1}^{j_*} V^n_j - T^n_{j-1} \ge \delta b_n^2 n^{9/16}  \right) \le \delta/8 ,
\end{align*}
and, since $j_* \delta b_n^2 n^{9/16}  \leq \delta u$, 
\begin{align}\label{d27.1}
P^0_\om \left(  \sum_{j=1}^{j_*} V^n_j - T^n_{j-1} \ge \delta u \right) \le \delta/8 . 
\end{align}

Recall $\calK (r)$ from \eqref{e:Hdef}.
Let
\begin{align*}
\wh V^n_k & = \inf\{t\geq V^n_{k}: X^{(n)}_t \in 
\Z^2 \setminus \calK(b_n n ^{3/8})\} \land T^n_k, \qquad k \geq 1,\\
\wt V^n_k & = \inf\{t\geq \wh V^n_{k}: |X^{(n)}_t - X^{(n)}(\wh V^n_k)| \geq (1/2)b_n n^{3/8} \} , \qquad k \geq 1. 
\end{align*}
We can use estimates for Brownian hitting 
probabilities \assmp to see that for some $c_2, c_3$ and $n_4$,
all $n\geq n_4$, $k$, 
\begin{align}\label{j2.13}
P^0_\om(\wh V^n_k < T^n_k \mid \calF_{V^n_k}) \geq c_2
\frac{\log (4 \beta_n) - \log (2 \beta_n)}
{\log (2 b_n n^{3/8})- \log (2 \beta_n)} 
\geq  c_3 /\log n . 
\end{align}
There exist \assmp $c_4$ and $n_5$, such that for
all $n\geq n_5$, $k\geq 2$, 
\begin{align*}
P^0_\om &(T^n_k - V^n_k \geq c_4 b_n^2 n^{3/4}  \mid \wh V^n_k < T^n_k, \calF_{\wh V^n_k}) \\
&\ge
P^0_\om(\wt V^n_k - \wh V^n_k \geq c_4 b_n^2 n^{3/4}  \mid \wh V^n_k < T^n_k, \calF_{\wh V^n_k}) \ge 3/4. 
\end{align*}
This and \eqref{j2.13} imply that the sequence $\{T^n_k - V^n_k\}_{k\geq 2}$ is stochastically  minorized by a sequence of i.i.d.~random variables which take value $c_4 b_n^2 n^{3/4} $ with probability $c_3 /\log n$ and they take value 0 otherwise.
This implies that for some $n_6$,
all $n\geq n_6$, $u \geq a_n^2$, 
\begin{align*}
P^0_\om\left( \frac 1{j_*} \sum_{j=2}^{j_*} T^n_j - V^n_j  \le  b_n^2 n^{3/4}/ \log^2 n \right) \le \delta/4 
\end{align*}
and, because  $j_* b_n^2 n^{3/4}/\log^2 n \geq u$ assuming $n_6$ is large enough,
\begin{align*}
P^0_\om\left(  \sum_{j=2}^{j_*} T^n_j - V^n_j  \le u \right) \le \delta/4 . 
\end{align*}
We combine this with \eqref{d27.1} and the definition of $\sigma^{n,2}_u$ to obtain
for some $n_7$,
all $n\geq n_7$, $u \geq a_n^2$, 
\begin{align}\label{d27.4}
P^0_\om(\sigma^{n,2}_u/u \le \delta) \geq 1-3\delta/8.
\end{align}
This completes the proof of the lemma.
\end{proof}

\bigskip

Let $Y^n_k = (Y^n_{k,1}, Y^n_{k,2}) = X^{(n)}(V^n_{k+1} ) - X^{(n)}(T^n_{k} )$.
Set 
$\bar Y^n_k = \sup_{T^n_k \leq t \leq V^n_{k+1}} |X^{(n)}(t)- X^{(n)}(T^n_{k} )|$.
For $x\in \Z^2$, let
 $\Pi_n(x) \in B'_n -u_n + \calO_n $ be the unique point with the 
property that $x-\Pi_n(x) = a_n y$ for some $y\in \Z^2$.

We next estimate the variance of  $X^{n,2}(V^n_{m+1}) = \sum_{k=0}^ m Y^n_k$.

\begin{lemma} \label{L:Ynest}
There exist $c_1, c_2$ and $n_1$ such that for all $n\geq n_1$, $k\geq 0$, $j=1,2$, and $\omega$,
\begin{align}\label{ma24.6}
E^0_\omega |Y^n_{k,j}| &\leq E^0_\omega |Y^n_k| \leq E^0_\omega |\bar Y^n_k| \leq c_1 \beta_n, \\
\label{ma24.7}
\Var Y^n_{k,j} &\le \Var \bar Y^n_{k} \le c_2 \beta_n^2 , \qquad \text { under } P^x_\omega.
\end{align} 
\end{lemma} 

\begin{proof} Let 
\begin{align}\label{n8.1}
\sX^{(n)}_k(t) &= X^{(n)}_t + \Pi_n(X^{(n)}(T^n_k)) - X^{(n)}(T^n_k),
\qquad t\in [T^n_k, V^n_{k+1}],
\end{align}
and note that 
\begin{align*}
Y^n_k = (Y^n_{k,1}, Y^n_{k,2}) = \sX^{(n)}_k(V^n_{k+1} ) - \sX^{(n)}_k(T^n_{k} ).
\end{align*}

It follows from the definition that we have $\sup_{S^n_{k-1} \leq t \leq U^n_k} |X^{(n)}(t ) - X^{(n)}(S^n_{k-1} )| \le 16\beta_n$, a.s.
This, \eqref{ma23.2} and the definition of $V^n_{k+1}$ imply that $|\bar Y^n_k|$ is stochastically majorized by an exponential random variable with mean $c_3 \beta_n$. This easily implies the lemma.
\end{proof}

Next we will estimate the covariance of $Y^n_{k,1}$ and $Y^n_{j,1}$ for $j\ne k$.

\begin{lemma} \label{L:covY}
There exist $c_1, c_2$ and $n_1$ such that for all $n\geq n_1$, $j < k-1$ and $\omega$ such that $0 \notin \Gamma^1_n \setminus \prt_i \Gamma^1_n$, under $P^0_\omega$,
\begin{align}\label{n8.3} 
\Cov(Y^n_{j,1},Y^n_{k,1}) & \le c_1 e^{-c_2 (k-j)} \beta_n^2.
\end{align}
\end{lemma}

\begin{proof}
Assume that $\omega$ is such that $0 \notin \Gamma^1_n \setminus \prt_i \Gamma^1_n$.
Let
\begin{align*}
\Gamma^3_n &= \Gamma^1_n \cap \ball(u_n + \calO_n, a_n/2)
= \ball(u_n + \calO_n, 2\beta_n),\\
\Gamma^4_n &= \prt_i \ball(u_n + \calO_n, 3\beta_n),\\
\tau(A) &= \inf\{t\geq 0: \sX^{(n)}_0(t) \in A\}.
\end{align*}

Suppose that $x,v\in \Gamma^3_n$ and $y \in \Gamma^4_n$. 
By the Harnack inequality proved in Lemma \ref{L:harn},
\begin{align}\label{ma23.4}
\frac{P_\omega^x (\sX^{(n)}_0(\tau(\Gamma^4_n)) = y)}
{P_\omega^v (\sX^{(n)}_0(\tau(\Gamma^4_n)) = y)}
\geq c_3.
\end{align}

Let $\calT^n_k$ have the same meaning as $T^n_k$ but relative to the process $\sX^{(n)}_k$ rather than $X^{(n)}$.
We obtain from \eqref{ma23.4} and the strong Markov property applied at $\tau(\Gamma^4_n)$ that, 
for any 
$x,v,y \in \Gamma^3_n$ we have
\begin{align*}
\frac{P_\omega^x (\sX^{(n)}_0(\calT^n_1) = y)}
{P_\omega^v (\sX^{(n)}_0(\calT^n_1) = y)}
\geq c_3.
\end{align*}
Recall that $T^n_0 =0$.
The last estimate implies that, for  
$x,v,y \in \Gamma^3_n$,
\begin{align*}
\frac{P_\omega (\sX^{(n)}_1(T^n_{1}) = y \mid \sX^{(n)}_0(T^n_0) = x)}
{P_\omega (\sX^{(n)}_1(T^n_{1}) = y \mid \sX^{(n)}_0(T^n_0) = v)}
\geq c_3.
\end{align*}
Since the process $X^{(n)}$ is time-homogeneous, this shows that for  
$x,v,y \in \Gamma^3_n$ and all $k$,
\begin{align}\label{d29.2}
\frac{P_\omega (\sX^{(n)}_{k+1}(T^n_{k+1}) = y \mid \sX^{(n)}_k(T^n_k) = x)}
{P_\omega (\sX^{(n)}_{k+1}(T^n_{k+1}) = y \mid \sX^{(n)}_k(T^n_k) = v)}
\geq c_3.
\end{align}
We now apply Lemma 6.1 of \cite{BTW} (see Lemma 1 of \cite{BK} for a better presentation of the same estimate) to see that \eqref{d29.2} implies that 
there exist constants $C_k$, $k\geq 1$, such that for every $k$ and all 
$x,v,y \in \Gamma^3_n$,
\begin{align*}
\frac{P_\omega^x (\sX_k^{(n)}(T^n_k) = y)}
{P_\omega^v (\sX_k^{(n)}(T^n_k) = y)}
\geq C_k.
\end{align*}
Moreover, $C_k\in(0,1)$, $C_k$'s depend only on $c_3$, and $1-C_k \le e^{-c_4 k}$ for some $c_4>0$ and all $k$.
By time homogeneity of $X^{(n)}$, for $m\leq j<k$ and all $x,v,y,z \in \Gamma^3_n$,
\begin{align*}
\frac{P^z_\omega (\sX_k^{(n)}(T^n_k) = y \mid \sX_j^{(n)}(T^n_j) = x)}
{P^z_\omega (\sX_k^{(n)}(T^n_k) = y \mid \sX_j^{(n)}(T^n_j) = v)}
\geq C_{k-j},
\end{align*}
and, by the strong Markov property applied at $T^n_j$, 
\begin{align*}
\frac{P^z_\omega (\sX_k^{(n)}(T^n_k) = y \mid \sX_j^{(n)}(T^n_j) = x)}
{P^z_\omega (\sX_k^{(n)}(T^n_k) = y \mid \sX_m^{(n)}(T^n_m) = v)}
\geq C_{k-j}.
\end{align*}
This and \eqref{ma24.6} imply that for $j<k-1$ and $x\in \Z^2$,
\begin{align}\nonumber
|\qE^x_\omega(Y^n_{k,1} - \qE^x_\omega Y^n_{k,1} \mid \calF_{T^n_{j+1}})|
&=
|\qE^x_\omega(Y^n_{k,1}  \mid \calF_{T^n_{j+1}})
- \qE^x_\omega Y^n_{k,1} |   \\
\nonumber
&\le  (1- C_{k-j-1}) \sup_{y\in \Z^2} \qE^y_\omega |Y^n_{k,1}| \\
&\le  e^{-c_4 (k-j-1)} c_5 \beta_n \leq 
 c_6 e^{-c_4 (k-j)}  \beta_n.
\label{ma24.8}
\end{align}
Hence  for $j< k-1$,
\begin{align*}
\Cov(Y^n_{j,1},Y^n_{k,1})
&= \qE^x_\omega ((Y^n_{j,1} - \qE^x_\omega Y^n_{j,1})
(Y^n_{k,1} - \qE^x_\omega Y^n_{k,1}))\\
&= \qE^x_\omega (\qE^x_\omega((Y^n_{j,1} - \qE^x_\omega Y^n_{j,1})
(Y^n_{k,1} - \qE^x_\omega Y^n_{k,1}) \mid \calF_{T^n_{j+1}}))\\
&= \qE^x_\omega ((Y^n_{j,1} - \qE^x_\omega Y^n_{j,1}) \qE^x_\omega(Y^n_{k,1} - \qE^x_\omega Y^n_{k,1} \mid \calF_{T^n_{j+1}}))\\
&\leq \qE^x_\omega (|Y^n_{j,1} - \qE^x_\omega Y^n_{j,1}| \cdot |\qE^x_\omega(Y^n_{k,1} - \qE^x_\omega Y^n_{k,1} \mid \calF_{T^n_{j+1}})|)\\
& \le 2 \qE^x_\omega |Y^n_{j,1}| c_6 e^{-c_4 (k-j)} \beta_n \\
& \le c_7 e^{-c_4 (k-j)} \beta_n^2.
\end{align*}
\end{proof}

\begin{proof}[Proof of Proposition \ref{d22.2}]
Assume that $\omega$ is such that $0 \notin \Gamma^1_n \setminus \prt_i \Gamma^1_n$.
We combine \eqref{n8.3} and \eqref{ma24.7} to see that for some $c_{1}$ and $c_{2}$ and all  $m\geq 1$, we have under $P^0_\om$, 
\begin{align}\label{d29.10}
\Var\left(\sum_{k=0}^m Y^n_{k,1}\right) &= 
\sum_{j=0}^m \sum_{k=0}^m
\Cov(Y^n_{j,1},Y^n_{k,1})\\
& \le \sum_{j=0}^m \sum_{k=0}^m
c_{1} e^{-c_{3} (k-j)} \beta_n^2 \leq c_{2} m \beta_n^2. \nonumber
\end{align}

For fixed $n$ and $\om$, the process 
$\{\sX^{(n)}_k(T^n_k), k\geq 1\}$ 
is Markov with a finite state space 
and one communicating class, so it has a unique stationary distribution. We will call it $\std(n)$.
We will argue 
that $\qE^{\std(n)}_\om Y^n_{k,1} = 0$. Since $X^{(n)}$ and $X^{(n-1)}$ 
satisfy the quenched invariance principle and they are random walks among symmetric (in distribution) conductances, they have zero means. Recall that 
$X^{(n)} = X^{n,1} + X^{n,2}$ and $\wh X^{n,1}$ has the same distribution as $X^{(n-1)}$. 
It follows that for some $c_{4}>0$ and $c_{5}<1/4$ and all large $t$, we have 
\begin{align*}
P^{\std(n)}_\om \left(\sup_{1\le s \le t} |\wh X^{n,1} _s| \geq  c_{4}\sqrt{t}\right) 
=P^{\std(n)}_\om \left(\sup_{1\le s \le t} |X^{(n-1)} _s| \geq  c_{4}\sqrt{t}\right) < c_{5}. 
\end{align*}
Since $\wh X^{n,1}_t =  X^{n,1}(\wh \sigma^{n,1}_t)$ and $\wh \sigma^{n,1}_t \geq t$, the last estimate implies that
\begin{align*}
P^{\std(n)}_\om \left(\sup_{1\le s \le t} |X^{n,1} _s| \geq  c_{4}\sqrt{t}\right)  < c_{5}. 
\end{align*}
We also have 
for some $c_{6}>0$ and $c_{7}<1/4$, and all large $t$, 
\begin{align*}
P^{\std(n)}_\om \left(\sup_{1\le s \le t} |X^{(n)} _s| \geq  c_{6}\sqrt{t}\right)  < c_{7}. 
\end{align*}
Since $X^{n,2} = X^{(n)} - X^{n,1}$, we obtain
for some $c_{8}>0$ and $c_{9}<1/2$ and all large $t$, 
\begin{align*}
P^{\std(n)}_\om \left(\sup_{1\le s \le t} |X^{n,2} _s| \geq  c_{8}\sqrt{t}\right)  < c_{9}. 
\end{align*}
This shows that $X^{n,2}$ does not have a linear drift.
It is clear from the law of large numbers that $\liminf_{t\to\infty} \sigma_t^{n,2}/t >0$, so $\wh X^{n,2}$ does not have a linear drift either.
We conclude that $\qE^{\std(n)}_\omega Y^n_{k,1} = 0$. 

Now suppose that $X^{(n)}_0$ does not necessarily have the distribution $\std(n)$. 
The fact that $\qE^{\std(n)}_\omega Y^n_{k,1} = 0$ and a calculation similar to that in \eqref{ma24.8} imply that,
\begin{align*}
|\qE^0_\omega  Y^n_{k,1}| \le  c_{10} e^{-c_{11} k}  \beta_n .
\end{align*}

Let $c_{12} $ be the constant denoted $c_1$ in \eqref{ma24.6}.
The last estimate and \eqref{ma24.6} imply that for some $c_{13}$ and all $m\geq 1$,
\begin{align} \nonumber
\left| E^0_\omega\sum_{k=0}^{m} Y^n_{k,1}\right| 
&\leq \sum_{k\geq 0} |\qE^0_\omega Y^n_{k,1}|
+ \sup_{k\geq 1} E^0_\om |\bar Y^n_k|   \\
\label{j1.1}
&\leq \sum_{k\geq 0} c_{10} e^{-c_{11} k}  \beta_n +c_{12}\beta_n
\leq  c_{13} \beta_n .
\end{align}
All estimates that we derived for $Y^n_{k,1}$'s apply to $Y^n_{k,2}$'s as well, by symmetry.

Note that $|X^{(n)}(U^n_{k+1} ) - X^{(n)}(T^n_{k} )|\geq \beta_n/2$.
We have
$V^n_{k+1}  - T^n_{k} \geq U^n_{k+1}  - T^n_{k} $ so
we can assume \assmp that $b_n/a_{n-1}$ is so large that for some $p_1>0$ and $n_{2}$, for all $n\geq n_{2}$ and $k\geq 1$, 
\begin{align*}
P_\omega^x(V^n_{k+1}  - T^n_{k} \geq \beta_n^2
\mid \calF_{T^n_k}) \geq p_1.
\end{align*}
Let $\calV_m$ be a binomial random variable with parameters 
$m$ and $p_1$. We see that $\sigma^{n,2}(V^n_{m })= \sum_{k=0}^m V^n_{k+1}  - T^n_{k}$ is stochastically minorized by 
$\beta_n^2 \calV_m$. 

Recall that $u \geq a_n^2$.
Let $m_1$ be the smallest integer such that
\begin{align}\label{c1.1}
P^0_\omega(V^n_{m_1 } \leq u) < \delta/4.
\end{align}
Then
\begin{align}\label{c1.2}
P^0_\omega(V^n_{m_1 -1} \leq u) \geq \delta/4.
\end{align}
Since $\delta$ in \eqref{d27.4} can be arbitrarily small, we have
for for some $n_{3}$ and all $n\geq n_{3}$, 
\begin{align}\label{c1.3}
P^0_\omega(\sigma^{n,2}_u/u \leq \delta^4) \geq 1-\delta/8.
\end{align}
The following estimate follows from the fact that $\sigma^{n,2}(V^n_{m_1-1 })$ is stochastically minorized by 
$\beta_n^2 \calV_{m_1-1}$, and from \eqref{c1.2}-\eqref{c1.3},
\begin{align*}
P^0_\omega(\beta_n^2 \calV_{m_1-1}  \leq \delta^4 u) 
&\geq 
P^0_\omega(\sigma^{n,2}(V^n_{m_1 -1}) \leq \delta^4 u) \\
&\geq
P^0_\omega(\sigma^{n,2}_u \leq \delta^4 u, V^n_{m_1 -1} \leq u) 
\geq \delta/8.
\end{align*}
This implies that for some $c_{14}$, we have
$m_1 \leq c_{14}\delta^3 u/\beta_n^2$. In other words, $u \geq m_1 \beta_n^2/(c_{14} \delta^3)$. 
Note that for a fixed $\delta$, we have for large $n$, \assmp $u^{1/2}\delta/4 -  c_{13} \beta_n \geq u^{1/2} \delta/8$.
These observations, \eqref{d29.10}, \eqref{j1.1} and the Chebyshev inequality imply that for  $m\le m_1$,
\begin{align}\label{d29.22}
P^0_\omega&\left(u^{-1/2}\left(\left|\sum_{k=0}^{m} Y^n_{k,1}\right|+\left|\sum_{k=0}^{m}Y^n_{k,2}\right|\right) \geq \delta/2\right)\\
&\le
P^0_\omega\left(\left|\sum_{k=0}^{m} Y^n_{k,1}\right| \geq u^{1/2} \delta/4\right)
+ P^0_\omega\left(\left|\sum_{k=0}^{m} Y^n_{k,2}\right| \geq u^{1/2} \delta/4\right)
\nonumber \\
&\le
P^0_\omega\left(\left|\sum_{k=0}^{m} Y^n_{k,1}
- E^0_\omega\sum_{k=0}^{m} Y^n_{k,1}\right| \geq u^{1/2} \delta/4
- c_{13} \beta_n\right)\nonumber \\
&\qquad + P^0_\omega\left(\left|\sum_{k=0}^{m} Y^n_{k,2}
- E^0_\omega\sum_{k=0}^{m} Y^n_{k,2}\right| \geq u^{1/2} \delta/4
- c_{13} \beta_n\right)
\nonumber \\
&\le  \frac{\Var\left(\sum_{k=0}^{m} Y^n_{k,1}\right)}
{u \delta^2/64} + \frac{\Var\left(\sum_{k=0}^{m} Y^n_{k,2}\right)}
{u \delta^2/64} \nonumber \\
&\le  \frac{2c_{2} m_1 \beta_n^2}
{(c_{14}^{-1}\delta^{-3} m_1\beta_n^2) \delta^2/64} \le c_{15} \delta.
\nonumber
\end{align}
Let $M = \min\{m\geq 1: 
u^{-1/2}\left(\left|\sum_{k=0}^{m} Y^n_{k,1}\right|+\left|\sum_{k=0}^{m}Y^n_{k,2}\right|\right)
\geq \delta\}$. By the strong Markov property applied at $M$ and \eqref{d29.22},
\begin{align}\label{d29.23}
&P^0_\omega\left(
\sup_{1\le m \le m_1}
u^{-1/2}\left(\left|\sum_{k=0}^{m} Y^n_{k,1}\right|+\left|\sum_{k=0}^{m}Y^n_{k,2}\right|\right)
\geq \delta,\  u^{-1/2}\left(\left|\sum_{k=0}^{m_1} Y^n_{k,1}\right|+\left|\sum_{k=0}^{m_1}Y^n_{k,2}\right|\right) \le \delta/2\right)\\
&\le 
P^0_\omega\left( u^{-1/2}\left(\left|\sum_{k=0}^{m_1-M} Y^n_{k,1}\right|+\left|\sum_{k=0}^{m_1-M}Y^n_{k,2}\right|\right) \geq \delta/2 \mid M < m_1\right)
\le c_{15} \delta. \nonumber
\end{align}

Recall that $u \geq m_1 \beta_n^2/(c_{14} \delta^3)$. For a fixed $\delta$ and large $n$, \assmp $u^{1/2}\delta - 2 c_{12} \beta_n \geq u^{1/2} \delta/2$.
It follows from this, \eqref{ma24.6} and \eqref{ma24.7} that
\begin{align}\label{ma25.1}
P^0_\om\left(\exists k \leq m_1: |\bar Y^n_k| \geq u^{1/2}\delta  \right) 
&\leq 
m_1 \sup_{k\leq m_1} P^0_\om\left( |\bar Y^n_k| \geq u^{1/2}\delta  \right) \\
&\leq 
m_1 \sup_{k\leq m_1} P^0_\om\left( |\bar Y^n_k| -E^0_\om |\bar Y^n_k|\geq u^{1/2}\delta - c_{12} \beta_n \right)\nonumber \\
&\le m_1 \frac {c_{11} \beta_n^2}{ u  \delta^2 /4}  
\le m_1 \frac{c_{11}  \beta_n^2}
{(c_{14}^{-1}\delta^{-3} m_1\beta_n^2) \delta^2} \le c_{16} \delta.
\nonumber
\end{align}

We use \eqref{c1.1}, \eqref{d29.22}, \eqref{d29.23} and \eqref{ma25.1} to obtain
\begin{align*}
&P^0_\omega \left(\sup_{0\le s \le u} u^{-1/2} |X^{n,2}_s| \geq 2\delta\right)\\
 &\le P^0_\omega(V^n_{m_1 } \le u) 
+ P^0_\omega\left(u^{-1/2}\left(\left|\sum_{k=0}^{m_1} Y^n_{k,1}\right|+\left|\sum_{k=0}^{m_1}Y^n_{k,2}\right|\right) \geq \delta/2\right)\\
& + P^0_\omega\left(
\sup_{1\le m \le m_1}
u^{-1/2}\left(\left|\sum_{k=0}^{m} Y^n_{k,1}\right|+\left|\sum_{k=0}^{m}Y^n_{k,2}\right|\right)
\geq \delta,\  u^{-1/2}\left(\left|\sum_{k=0}^{m_1} Y^n_{k,1}\right|+\left|\sum_{k=0}^{m_1}Y^n_{k,2}\right|\right) \le \delta/2\right)\\
& +
P^0_\om\left(\exists k \leq m_1: |\bar Y^n_k| \geq u^{1/2}\delta  \right) \\
&\le \delta/4 + c_{15} \delta + c_{15} \delta + c_{16}\delta.
\end{align*}
Since $\delta>0$ is arbitrarily small, this implies that for every $\delta>0$, some $n_{3}$ and all $n\geq n_{3}$,
\begin{align*}
P^0_\omega &\left(\sup_{0\le s \le u} u^{-1/2} |X^{n,2}_s| \geq \delta\right)\le \delta/2.
\end{align*}
This and \eqref{n8.5} yield the proposition.
\end{proof}

Recall from \eqref{e:bfPdef} the definition of the averaged measure $\bfP$.

\begin{lemma}\label{n9.1}
For every $\delta>0$ there exists $n_1$ such that for all $n\geq n_1$ and  $u\geq a_n^2$,  
\begin{align}\label{n9.2}
\bfP
\left(  \sigma^{n,2}_u / u \le \delta, \sup_{0\le s \le u} u^{-1/2} |X^{n,2}_s| \le \delta \right) \geq 1-\delta.
\end{align}
\end{lemma}

\begin{proof}
By Proposition \ref{d22.2} applied to $\delta/2$ in place of $\delta$, for every $\delta>0$ there exists $n_2$ such that for all $n\geq n_2$,  $u\geq a_n^2$, and $\omega$ such that $0 \notin \Gamma^1_n \setminus \prt_i \Gamma^1_n$, 
\begin{align}\label{n9.3}
P^0_\omega
\left(  \sigma^{n,2}_u / u \le \delta, \sup_{0\le s \le u} u^{-1/2} |X^{n,2}_s| \le \delta \right) \geq 1-\delta/2.
\end{align}

Let $|A|$ denote the cardinality of $A\subset \Z^2$. Since $|\Gamma^1_n| \leq 25 \beta_n^2 \leq 25 a_n^2 n^{-1/2} = 25 n^{-1/2} |B'_n|$, the definitions of $\sO_n$ and $\Gamma^1_n$ imply that $\bfP(0 \in \Gamma^1_n \setminus \prt_i \Gamma^1_n) < \delta/2$ for some $n_3 \geq n_2$ and all $n \geq n_3$. This and \eqref{n9.3} imply \eqref{n9.2}.
\end{proof}

In the following lemma and its proof, when we write the Prokhorov distance between processes such as $\{ (1/a)X^{(n-1)}_{ta^2}, t\in[ 0,1]\}$, we always assume that they are distributed according to $\bfP$.

\begin{lemma}\label{d22.1}
There exists a function $\rho^*: (0,\infty) \to (0,\infty)$ with 
$ \lim_{\delta\downarrow 0} \rho^*(\delta) = 0$ 
and a sequence $\{a_n\}$ with the following properties,
\begin{align}\label{d19.3}
&\dP(\{ (1/a)X^{(n-1)}_{ta^2}, t\in[ 0,1]\}, \PBM) \le 2^{-n},\qquad a \geq a_n.
\end{align}
Moreover, suppose that for $\delta<1/2$ and all $u\geq a_n^2$, 
\begin{align}\label{d19.2}
\bfP \left(  \sigma^{n,2}_u / u \le \delta, \sup_{0\le s \le u} u^{-1/2} |X^{n,2}_s| \le \delta \right) \geq 1-\delta.
\end{align}
Then
$\dP( \{(1/a)X^{(n)}_{ta^2}, t\in[ 0,1]\}, \PBM) \le 2^{-n} + \rho^*(\delta)$, for all $a\geq a_n$.
\end{lemma}

\begin{proof}

Formula \eqref{d19.3} is  special case of \eqref{e:PdistBM}.

Fix some $a\geq a_n$. We will apply \eqref{d19.2} with $u=a^2$. 
Note that on the event in \eqref{d19.2} we have
\begin{align}\label{n9.5}
1- \sigma^{n,1}_{a^2}/a^2= u/u- \sigma^{n,1}_u/u 
= \sigma^{n,2}_u / u \le \delta.
\end{align}
The function $t\to \sigma^{n,1}_{ta^2}/a^2$ is Lipschitz with the constant 1 and $\sigma^{n,1}_{ta^2}/a^2 \leq t$ so \eqref{n9.5} implies for $t\in[0,1]$,
\begin{align}\label{n9.6}
t- \sigma^{n,1}_{ta^2}/a^2 \leq 1- \sigma^{n,1}_{a^2}/a^2 \leq \delta.
\end{align}

Recall the function $\rho(\delta)$ from the proof of Lemma \ref{d21.1}, 
such that $\PBM(\Osc(W,\delta) \geq\rho(\delta) )<\rho(\delta)$ and 
$\lim_{\delta\downarrow 0} \rho(\delta) = 0$.
By \eqref{n9.6}, we can apply Lemma \ref{d21.1} with $\sigma_t = \sigma^{n,1}_{ta^2}/a^2$. Recall
that  $W^*(t) = W(\sigma_t)$.
By the definition of $\wh X^{n,1}$,
\begin{align}
\nonumber
&\dP( \{(1/a) X^{n,1}_{ta^2}, t\in[0,1]\}, \PBM) \\
\nonumber
&\le  \dP( \{(1/a) X^{n,1}_{t/a^2}, t\in[0,1]\}, \{W^*_t, t\in[0,1]\})
 + \dP( \{W^*_t, t\in[0,1]\}, \PBM) \\ 
\nonumber 
&\le  \dP( \{(1/a) X^{n,1}_{ta^2}, t\in[0,1]\}, \{W^*_t, t\in[0,1]\})
 + \rho(\delta) + \delta\\ 
&= \dP( \{(1/a)\wh X^{n,1}(\sigma^{n,1}_{ta^2}), t\in[0,1]\}, \{W(\sigma^{n,1}_{ta^2}/a^2), t\in[0,1]\})
 + \rho(\delta) + \delta  .
\end{align}

Recall from \eqref{e:Xhatdsn} that for a fixed $\omega\in \Omega$, the distribution of $\{\wh X^{n,1}_t, t\geq 0\}$ 
is the same as that of $\{X^{n-1}_t, t\geq 0\}$.
In view of Theorem \ref{T:eK}, we can make $a_n$ so large \assmp that 
$\Pp(\Osc(\wh X^{n,1},\delta) \geq 2\rho(\delta) )< 2\rho(\delta)$. 
This, Lemma \ref{d21.2} and the definition of the Prokhorov distance imply that
\begin{align*}
\dP( \{(1/a)\wh X^{n,1}&(\sigma^{n,1}_{ta^2}), \, t\in[0,1]\}, \{W(\sigma^{n,1}_{ta^2}/a^2), t\in[0,1]\})  \\
&\le 
\dP( \{(1/a)\wh X^{n,1}_{ta^2}, t\in[0,1]\}, \{W_{t}, t\in[0,1]\}) + 4\rho(\delta) \\
&= \dP( \{(1/a) X^{(n-1)}_{ta^2}, t\in[0,1]\}, \{W_{t}, t\in[0,1]\}) + 4\rho(\delta) \\
&\le  2^{-n}  + 4\rho(\delta).
\end{align*}
In the final two lines line we used  \eqref{e:Xhatdsn} and \eqref{d19.3}.

Combining the estimates above, since
$P^0_\om \left(  \sup_{0\le s \le u} u^{-1/2} |X^{n,2}_s| \le \delta \right) \geq 1-\delta$ and 
$X^{(n)} = X^{n,1} + X^{n,2}$, Lemma \ref{ma26.5} shows that
\begin{align*}
&\dP( \{(1/a) X^{(n)}_{ta^2}, t\in[0,1]\}, \PBM) \\
&\le
\dP( \{(1/a) X^{(n)}_{ta^2}, t\in[0,1]\}, \{(1/a) X^{n,1}_{ta^2}, t\in[0,1]\}) \\
&\qquad + \dP( \{(1/a) X^{n,1}_{ta^2}, t\in[0,1]\}, \PBM) \\
&\le 
\delta + 2^{-n}  + 5 \rho(\delta) + \delta .
\end{align*}
We conclude that the lemma holds if we take $\rho^*(\delta) =  5\rho(\delta) + 2\delta $.
\end{proof}

\begin{proof}[Proof of Theorem \ref{T:main}]
Choose an arbitrarily small $\eps>0$. We will show that there exists $a_*$ such that for every $a\geq a_*$, 
\begin{align}\label{d23.1}
&\dP( \{(1/a) X_{ta^2}, t\in[0,1]\}, \PBM) \le \eps .
\end{align}

Recall $\rho^*$ from Lemma \ref{d22.1}.
Let $n_1 $ be such that $2^{-n_1}\le \eps/4$ and let $\delta>0$ be so small that 
$2^{-n_1}+ \rho^*(\delta) < \eps/2$. Let $n_2$ be defined as $n_1$ in Lemma \ref{n9.1}, 
relative to this $\delta$. Then, according to Lemma \ref{d22.1}, 
\begin{align}\label{d22.5}
\dP( \{(1/a)X^{n}_{ta^2}, t\in[ 0,1]\}, \PBM) \le 2^{-n} + \rho^*(\delta) < \eps/2,
\end{align}
for all $n\geq n_3 := n_1\lor n_2$ and $a\geq a_n$.

For a set $K$ let 
$\calB(K,r) = \{z: \dist(z,K) < r\}$ and recall the definition of $D_n$ given in \eqref{ma26.1}. Let 
\begin{align*}
F_1 &= \{0\in \calB( D_{n+1}, a_{n+1}/\log (n+1))\},\\
F_2 &= \{0\notin \calB( D_{n+1}, a_{n+1}/\log (n+1))\}
\cap
\{\exists t\in[0, a_{n+1}^2]: X^{(n)}_t \in  D_{n+1}\},\\
G_1^k &= \{0\in \calB( D_k, b_k/k)\},
\qquad k> n+1,\\
G_2^k &= \{0\notin \calB( D_k, b_k/k)\}
\cap
\{\exists t\in[0, a_{n+1}^2]: X^{(n)}_t \in  D_k\}, \qquad k> n+1.  
\end{align*}

The area of  $\calB(D_{n+1}, a_{n+1}/\log (n+1))$ is bounded by $c_1 (a_{n+1}/\log (n+1))^2$ so 
\begin{align}\label{d23.10}
\Pp(F_1) \le c_1 (a_{n+1}/\log (n+1))^2/ a_{n+1}^2 = c_1 /\log^2 (n+1).
\end{align}
We choose $n_4 > n_3 $ such that 
$c_1 /\log^2 (n+1) < \eps/8$ for $n \geq n_4$.

Note that $D_{n+1}$ is a subset of a square with side $4\beta_{n+1} \leq 4 a_{n+1} n^{-1/4}$. This easily implies that
there exists $n_5 \geq n_4$ such that for $n\geq n_5$,
\begin{align*}
\PBM\left(\exists t\in[0, a_{n+1}^2]: W(t) \in  D_{n+1} 
\mid 
0\notin \calB( D_{n+1}, a_{n+1}/\log (n+1))
\right) \le \eps/16.
\end{align*}
We can assume \assmp that $a_{n+1}/a_n$ is so large that
for some $n_6 \geq n_5$ and all $n\geq n_6$,  
\begin{align}\label{d23.11}
\Pp(F_2) &\le \Pp\left(\exists t\in[0, a_{n+1}^2]: X^{(n)}_t \in  D_{n+1} 
 \mid  0\notin \calB( D_{n+1}, a_{n+1}/\log (n+1)) \right)   \\
 &\le \eps/8.
\end{align}

The area of  $\calB( D_k, b_k/k)$ is bounded by $c_2 b_k^2/k$ so 
\begin{align}\label{d23.12}
\Pp(G_1^k) \le (c_2 b_k^2/k)/ a_k^2 \le c_3 ( b_k^2/k)/ (k b_k^2)= c_3 /k^2.
\end{align}
We let $n_7 >n_6$ be so large that $\sum_{k\geq n_7} c_3 /k^2 < \eps/8$.
For all $k>n+1\geq n_7+1$, we make $b_k/k$ so large \assmp that 
\begin{align}\label{d23.13}
\Pp(G_2^k) \le 
\Pp\left(\sup_{t\in[0, a_{n+1}^2]} |X^{n}_t| \geq b_k/k\right) \le c_3/k^2.
\end{align}

We combine \eqref{d23.10}, \eqref{d23.11}, \eqref{d23.12} and \eqref{d23.13} to see that for $n\geq n_7$,
\begin{align}\label{d23.14}
\Pp&(\exists t\in[0, a_{n+1}^2] \ \exists k\geq n+1: X^{(n)}_t \in  D_k)\\
&\le
\Pp(F_1) +  \Pp(F_2) + \sum_{k>n+1}
\Pp(G_1^k) + \sum_{k>n+1}
\Pp(G_2^k) \nonumber\\
& \le \eps/8 + \eps/8 + \eps/8 + \eps/8 = \eps/2.
\nonumber
\end{align}

Let $R_{n+1} = \inf\{t\geq 0: X_t \in \bigcup_{k\geq n+1} \calD_k\}$. 
It is standard to construct $X$ and $X^{(n)}$ on a common probability space so that $X_t = X_t^n$ for all $t\in [0, R_{n+1})$. This and \eqref{d23.14}
imply that for $n\geq n_7$ and all $a\in[a_n, a_{n+1}]$ we have
\begin{align*}
P(\exists t\in[ 0,1]: (1/a)X_{ta^2} \ne (1/a)X^{(n)}_{ta^2}) \le \eps/2.
\end{align*}
We combine this with \eqref{d22.5} to see that for all $a\geq a_{n_6}$, 
\begin{align*}
\dP( \{(1/a)X_{ta^2}, t\in[ 0,1]\}, \PBM) \le  \eps/2 +\eps/2 = \eps.
\end{align*}
We conclude that \eqref{d23.1} holds with $a_* = a_{n_7}$. 

This completes the proof of AFCLT. The WFCLT then follows from Theorem 2.13 of \cite{BBT1}.
\end{proof}

\end{document}